\documentclass[11pt,a4paper]{amsart}
\usepackage{amsfonts}
\usepackage{amsthm}
\usepackage{amsmath}
\usepackage{amssymb}
\usepackage{amscd}
\usepackage{t1enc}
\usepackage[margin=2.9cm]{geometry}
\usepackage{epstopdf}
\usepackage{times}
\usepackage{enumerate}
\usepackage{changepage}
\usepackage{mfirstuc}
\usepackage{tikz}
\usepackage{cite}
\usepackage{booktabs}
\usepackage{tabularx}
\numberwithin{equation}{section}

\theoremstyle{plain}
\newtheorem{theorem}{Theorem}[section]
\newtheorem{lemma}[theorem]{Lemma}

\newtheorem{cor}[theorem]{Corollary}
\newtheorem{prop}[theorem]{Proposition}

\newtheorem{conjecture}[theorem]{Conjecture}
\theoremstyle{remark}
\newtheorem{remark}[theorem]{Remark}
\newtheorem{definition}[theorem]{Definition}

\newcommand{\F}{\mathbb{F}}
\newcommand{\C}{\mathbb{C}}

\newcommand{\Nm}{\operatorname{N}}
\newcommand{\Tr}{\operatorname{Tr}}

\usepackage{xcolor}
\definecolor{eyecolor}{RGB}{204,232,207}
\pagecolor{eyecolor!90}

\begin{document}

\title{Norm-Trace and Kloosterman sums in finite semi-simple algebras}

\author{Daqing Wan}
\address{Center for Discrete Mathematics, College of Mathematics and Statistics, Chongqing University, Chongqing 401331, China.}

 \email{dwan@math.uci.edu}



\begin{abstract}
        An asymptotic formula with a square root error term is obtained for the number of elements with given trace and norm in
        a finite semisimple algebra over a finite field. This extends previous results
        from finite \'{e}tale algebras (commutative case) to finite  semi-simple algebras (non-commutative case). The main idea is to apply Eichler's formula for Gauss sums over the general linear group and the Hasse-Davenport relation to reduce the problem to the classical geometric case where the result is known to be true.
As an application of this reduction, we also obtain a square root estimate for Kloosterman sum over semi-simple
        algebras. Similar square root estimates are discussed when norm-trace is replaced by product-trace, leading to a new conjecture on product-trace counting over finite semi-simple algebras.

        \end{abstract}

\maketitle
\setcounter{tocdepth}{1}

\section{Introduction}

Let $\F_q$ be the finite field of $q$ elements with characteristic $p$. By the Artin-Wedderburn theorem,
any finite  semi-simple algebra $B$ over $\F_q$ of degree $n\geq 2$ is isomorphic to a finite product of matrix algebras over $\F_q$.
Namely,
\begin{equation}\label{B}
B = M_{d_1}(\F_{q^{n_1}}) \times \cdots \times M_{d_k}(\F_{q^{n_k}}), \ \ n_1d_1^2+\cdots + n_kd_k^2 = n,\ |B| =q^n,
\end{equation}
where $\F_{q^n_i}$ denotes the finite field extension of $\F_q$ with degree $n_i$, and $M_{d}(\F_q)$ denotes the $d\times d$ matrix algebra over $\F_q$.
This factorization of $B$ is uniquely determined  up to permutations. The group of invertible elements in $B$ is
\begin{equation}
B^* = {\rm GL}_{d_1}(\F_{q^{n_1}}) \times \cdots \times {\rm GL}_{d_k}(\F_{q^{n_k}}), \ \ |B^*| = \prod_{i=1}^k |{\rm GL}_{d_i}(\F_{q^{n_i}})|.
\end{equation}
The semi-simple algebra $B$ is called  \'{e}tale (or commutative) if all $d_i=1$.

For $x \in B$, multiplication by $x=(A_1, \cdots, A_k)$ on the left acts on the reduced $\F_q$-vector space 
$$R_B: = \bigoplus_{i=1}^k \F_{q^{n_i}}^{d_i}, \ \ \ \ \dim_{\F_q}(R_B) = \sum_{i=1}^k d_in_i.$$
Its trace ${\rm Tr}_B(x)$ and norm ${\rm N}_B(x)$, called reduced trace 
and reduced norm in the literature, are well defined elements in $\F_q$.  Clearly, the trace map
${\rm Tr}_B: B \rightarrow \F_q$
is a surjective homomorphism under addition, and thus for $a\in \F_q$, we have
$$|{\rm Tr}_B^{-1}(a)| =\frac{|B|}{q} = q^{n-1} = \frac{|B^*|}{q} +O_n(q^{n-2}).$$
As will be seen in Proposition \ref{prop-tr}, it is also easy to see that 
$$|B^*\cap {\rm Tr}_B^{-1}(a)| =\frac{|B^*|}{q} +O_n(q^{n-2}).$$
Similarly, the norm map
${\rm N}_B: B^* \rightarrow \F_q^*$ is a surjective homomorphism under multiplication, and thus for $b\in \F_q^*$, we have
$$|{\rm N}_B^{-1}(b)| =\frac{|B^*|}{q-1} = \frac{1}{q-1} \prod_{i=1}^k \prod_{j=0}^{d_i-1}(q^{d_i n_i}-q^{jn_i}).$$
The counting problem becomes more interesting if we consider the addition and multiplication simultaneously.
For any given $a\in \F_q$ and $b\in \F_q^*$, let
$$\Nm_B(a, b)= |{\rm Tr}_B^{-1}(a) \cap {\rm N}_B^{-1}(b) | = \# \{ x \in B^* | {\rm Tr}_B(x)=a,  {\rm N}_B(x)=b\}$$
denote the number of elements in $B^*$ with trace $a$ and norm $b$. Note that if ${\rm N}_B(x)=b \in \F_q^*$, then $x \in B^*$.
Heuristically, as seen above, the trace map and the norm map when restricted to $B^*$ are expected to be increasingly independent as $q$ grows. This suggests that $\Nm_{B}(a,b)$ should be roughly $|B^*|/q(q-1)$ with hopefully a square
root error term. Namely, one hopes that
$$\Nm_{B}(a,b) = \frac{|B^*|}{q(q-1)} + O_n(q^{\frac{n-2}{2}}). $$
The main result of this note is to prove such an estimate. In fact, we prove the following more precise version.

\begin{theorem}\label{thm} Let $B$ be a finite semi-simple algebra over $\F_q$ of degree $n=\sum_{i=1}^k d_i^2 n_i \geq 2$ as given in equation (\ref{B}).   For $a \in \F_q$ and $b\in\F_q^*$, we have
        \begin{align*}
                \left| \Nm_B(a, b) -
                \left( \frac{|B^*|}{q(q-1)}
                +\frac{(-1)^{\sum_i d_i}q^{\sum_i n_i{d_i\choose 2}}}{q} \right) \right|
                \leq \bigg(\sum_i d_in_i  -1\bigg)q^{\frac{n-2}{2}} .
        \end{align*}
        \end{theorem}
Note that the second part in the main term is bounded by
$$q^{\frac{\sum_i (n_id_i^2 - n_i d_i)}{2}-1} \leq      q^{\frac{n-3}{2}}.$$
It follows that a slightly weaker but simpler looking version of Theorem \ref{thm} is
\begin{cor}\label{cor}
Let $B$ be a finite semi-simple algebra over $\F_q$ of degree $n=\sum_{i=1}^k d_i^2 n_i \geq 2$ as given in equation (\ref{B}).
For $a \in \F_q$ and $b\in\F_q^*$, we have
        \begin{align}
                \left| \Nm_B(a, b) -
                \frac{|B^*|}{q(q-1)}  \right|
                \leq \bigg(\sum_{i=1}^k d_in_i \bigg)q^{\frac{n-2}{2}} \leq n q^{\frac{n-2}{2}}.
        \end{align}

\end{cor}

\begin{remark}
The constant $\sum_{i=1}^k d_in_i$ attaines its upper bound $n$ in the commutative case, but is strictly
smaller than $n$ in the non-commutative case when some $d_i>1$. In the special case when $B = M_{d}(\F_{q})$, the number
$\Nm_B(a, b)$ is studied in Kim \cite[Theorem 6.1]{Ki} via the Bruhat decomposition, but without giving an estimate.
\end{remark}

We now recall earlier results in the commutative case.
When $B= \F_{q^{n}}$ is a field (thus $k=1$, $d_1=1$ and $n_1=n$),
using exotic Kloosterman sheaf, Katz \cite[Theorem 4]{Ka} proved the estimate
\begin{align}\label{Ka}
\left|\Nm_{\F_{q^n}}(a,b) -\frac{q^{n}-1}{q(q-1)}\right| \leq nq^{\frac{n-2}{2}}.
        \end{align}
for $a\in \F_q, b \in {\F^*_q}$ and $n\geq 2$. This is exactly the estimate in Corollary \ref{cor} when $B= \F_{q^n}$.

In the geometric case when $B= \F_q^n$ ( a product of $ n$ copies of $\F_q$), it is clear that $\Nm_{\F_q^n}(a, b)$ is equal to the number of $\F_q$-rational points on
the toric Calabi-Yau hypersurface
$$Y_1 +\cdots + Y_{n-1} +\frac{b}{Y_1\cdots Y_{n-1}} -a =0,$$
which can be singular for some $a \in \F_q$.
The cohomology of this family of toric hypersurfaces is studied in great detail in Rojas-Leon and Wan \cite{RW07}, which ultimately depends on Kloosterman
sheaf.
As a consequence \cite[Corollary 5.2]{Wa}\cite[Theorem 2.2]{MW10}, for all $a\in \F_q$,
we have
\begin{align}\label{imp'}
|\Nm_{\F_q^n}(a, b) - \frac{(q-1)^{n-1}+(-1)^{n}}{q} | & \leq (n-1)q^{\frac{n-2}{2}}.
\end{align}
This is exactly the bound in Theorem \ref{thm} when $B= \F_{q}^n$. This estimate has applications in coding theory, see Kim \cite{Ki11}.

The two numbers $\Nm_{\F_{q^n}}(a,b)$ and $\Nm_{\F_{q}^n}(a,b)$ are closely related.
It is shown in Moisio and Wan \cite{MW10} via the Hasse-Davenport relation that
\begin{align}
\Nm_{\F_{q^n}}(a,b) -\frac{q^{n-1}-1}{q-1} &= (-1)^{n-1} (\Nm_{\F_q^n}(a, b) - \frac{(q-1)^{n-1}+(-1)^{n}}{q}).
\end{align}
This equation together with estimate (\ref{imp'}) slightly improves the Katz bound (\ref{Ka}) to
\begin{align}\label{imp}
                \left|\Nm_{\F_{q^n}}(a,b) -\frac{q^{n-1}-1}{q-1}\right| \leq
                (n-1)q^{\frac{n-2}{2}}.
        \end{align}
In the case that $n$ is a power of $p$, the estimate (\ref{imp}) is first proved in Moisio \cite{Mo} using Deligne's estimate for hyper-Kloosterman sums. Note that the improved estimate (\ref{imp}) is exactly the estimate in Theorem \ref{thm} when $B=\F_{q^n}$.
The bound (\ref{imp}) is used in the recent study of locally recovery codes by Matthews-Morrison-Murphy \cite{MMM}.

The estimates (\ref{imp'}) and (\ref{imp}) for $B = \F_q^n $ and $\F_{q^n}$ are unified and extended in Lin-Wan \cite{LW} to
any finite \'{e}tale algebra $B$ of degree $n$ over $\F_q$, i.e., when all $d_i=1$. Namely, the following result is proved
in \cite[Theorem 1.3]{LW}.

\begin{prop}\label{thm1} Let $B=\prod_{i=1}^k \F_{q^{n_i}}$ be an \'{e}tale algebra over $\F_q$ of degree $n=\sum_i n_i\geq 2$.
For $a \in \F_q$ and $b\in\F_q^*$, we have
        \begin{align*}
                \left| \Nm_B(a, b) -
                \left( \frac{\prod_{i=1}^k(q^{n_i}-1)}{q(q-1)}
                +\frac{(-1)^{k}}{q} \right) \right|
                \leq (n-1)q^{(n-2)/2}.
        \end{align*}
        \end{prop}

Two proofs of this result are given in  \cite{LW}. The first proof
is again to use the Hasse-Davenport relation to prove the following comparison relation for $B=\prod_{i=1}^k \F_{q^{n_i}}$:
\begin{align}\label{rel}
                                       \Nm_B(a, b) -
                \bigg(\frac{\prod_{i=1}^k(q^{n_i}-1) }{q(q-1)} +\frac{(-1)^k}{q}\bigg)
                &= (-1)^{n-k} (\Nm_{\F_q^n}(a, b) -
                \frac{(q-1)^{n-1} +(-1)^n}{q} ).
                \end{align}
Hence Proposition \ref{thm1} is also reduced to the estimate (\ref{imp'}). The second proof is to apply $\ell$-adic cohomology
and exotic Kloosterman sheaf to extend Katz's method.

To further extend Proposition \ref{thm1} to the general non-commutative case, one could try any of the above two approaches in the commutative
case.  Here in this note, we only try the direct reduction approach. In addition to the Hasse-Davenport relation,  a key new ingredient is
the Eichler formula for Gauss sums over the general linear group.  The following reduction lemma extends the comparison relation (\ref{rel}) to non-commutative case.

\begin{lemma}\label{lem0}
Let $B$ be a finite semi-simple algebra over $\F_q$ of degree $n=\sum_{i=1}^k d_i^2 n_i \geq 2$ as given in equation (\ref{B}) with $m=\sum_{i=1}^k d_in_i$.
For $a \in \F_q$ and $b\in\F_q^*$, we have
$$
                                       \Nm_B(a, b) -
               \left (\frac{|B^*|}{q(q-1)} +\frac{(-1)^{\sum_i d_i}q^{\frac{n-m}{2}}}{q}\right)
                = (-1)^{m-\sum_i d_i} q^{\frac{n-m}{2}}\left(\Nm_{\F_q^m}(a, b) -
                \frac{(q-1)^{m-1} +(-1)^m}{q} \right).
                $$
                \end{lemma}
It is clear that this reduction lemma and  the estimate (\ref{imp'}) together yield Theorem \ref{thm}. It might be of interest to give a cohomological
interpretation of the reduction in Lemma \ref{lem0}.

The proof of our reduction can be used to study Kloosterman sum over the semi-simple algebra $B$ over $\F_q$.
Let $\psi: \F_q \longrightarrow \C$ be a fixed non-trivial additive character.
For $b\in \F_q^*$,  the Kloosterman sum over $B$ is defined by
$$K_{B}(b): = \sum_{x \in B^*, \ {\rm N}_{B}(x)=b} \psi({\rm Tr}_B(x)).$$
As a consequence of our reduction, we prove

\begin{theorem}\label{K}
Let $B$ be a finite semi-simple algebra over $\F_q$ of degree $n=\sum_{i=1}^k d_i^2 n_i \geq 2$ as given in equation (\ref{B}).
Let $\psi$ be a non-trivial additive character of $\F_q$. For $b\in\F_q^*$, we have
        \begin{align}\label{KL}
                \left| K_B(b): =  \sum_{x \in B^*, \ {\rm N}_{B}(x)=b} \psi\left({\rm Tr}_B(x)\right) \right|
                \leq (\sum_{i=1}^k d_in_i )q^{\frac{n-1}{2}} \leq nq^{\frac{n-1}{2}}.
        \end{align}
        \end{theorem}

\begin{remark}
 In the geometric case $B=\F_q^n$,  the estimate (\ref{KL}) becomes Deligne's classical bound \cite[Section 7.1]{De} for
hyper-Kloosterman sum:
\begin{align}\label{K-clas}
                \left| K_{\F_q^n}(b): =  \sum_{x_i \in \F_q^*, \ x_1\cdots x_n=b} \psi\big(x_1+\cdots +x_n\big) \right|
                \leq nq^{\frac{n-1}{2}}, \ \ b \in \F_q^*.
        \end{align}
More generally, when $B=\prod_{i=1}^k \F_{q^{n_i}}$ is finite \'{e}tale (commutative), Deligne \cite[Equation 7.2.5]{De} gave the reduction formula
\begin{align}\label{red}
K_B(b) = (-1)^{n-k} K_{\F_q^n}(b),
 \end{align}
 with two different proofs, one via the Hasse-Devenport relation and another one via $\ell$-adic cohomology \cite[Proposition 7.20]{De}.
 Thus, the same bound
 $$\left| K_{\prod_{i=1}^k \F_{q^{n_i}}}(b) \right| \leq nq^{\frac{n-1}{2}}$$
 holds, agreeing with the estimate (\ref{KL}) when $B$ is finite \'etale.

 For general semi-simple algebra $B$ over $\F_q$ as in
 Theorem \ref{K},  we shall prove the reduction
 \begin{align}\label{kl6}
K_{B}(b) =  (-1)^{\sum_i d_i(n_i-1)}q^{{\sum_i n_i {d_i \choose 2}}} K_{\F_q^m}(b).
                \end{align}
This is the Kloosterman sum analogue of the norm-trace reduction in Lemma \ref{lem0}.
This and Deligne's  estimate (\ref{K-clas}) together give Theorem \ref{K}.
It might be of interest to give a cohomological interpretation of the reduction (\ref{kl6}).

In the special case $B= M_{d}(\F_q$) (thus, $k=1$, $n_1=1$, $n=d^2$), the reduction (\ref{kl6}) becomes
$$K_{B}(b) = \sum_{x \in {\rm GL}_d(\F_q), \ {\rm det}(x)=b} \psi({\rm Tr}(x)) =q^{{d\choose 2}}K_{\F_q^d}(b),$$
recovering the reduction formula in Kim \cite[Corollary 5.2]{Ki} obtained via
the Bruhat decomposition. See also Li-Hu \cite{LH}
 for a simpler reduction via an average argument over the Borel subgroup.
\end{remark}

                In the last section, motivated by recent work of Zelingher \cite{Ze} on matrix Kloosterman sum, we discuss a product-trace version of Kloosterman  sum and its counting analogue over semi-simple algebra $B$.

\section{Proof of Theorem $\ref{thm}$}

Let $B$ be a finite semi-simple algebra over $\F_q$ of degree $n$ as given in equation (\ref{B}).
The idea  to prove Theorem \ref{thm} is to use the standard character sum argument together with Eichler's formula \cite{EI} for
Gauss sums over general linear groups to reduce the problem to the geometric case $B=\F_q^m$ for which the theorem is known to be true.

Let $ \Tr_{n_i}=\Tr_{\F_q^{n_i} / \F_q}$ be the trace map from $\F_{q^{n_i}}$ to $\F_q$. Similarly, let $\Nm_{n_i}=\operatorname{Norm}_{\F_{q^{n_i}} / \F_q}$ be the norm map from $\F_{q^{n_i}}$ to $\F_q$. Recall that
\begin{align*}
        \Nm_{B}(a,b)=
        \#\left\{ (x_1,\cdots, x_k)\ \left| \ \ x_i\in M_{d_i}(\F_{q^{n_i}}), \ \
        \sum_{i=1}^{k}\Tr_{n_i}({\rm tr}(x_i))=a,\ \ \prod_{i=1}^{k}\Nm_{n_i}({\rm det}(x_i))=b
        \right. \right\},
\end{align*}
where ${\rm tr}(x_i)$ and ${\rm det}(x_i)$ denote the trace and determinant of the $d_i\times d_i$ matrix $x_i$.

In the introduction, we assumed that $b\in \F_q^*$ which implies that $x \in B^*$ if $\Nm_B(x)=b$.
For completeness, we now also allow $b=0$.

We first treat this easy case with $b=0$.
An element $x=(x_1, \cdots, x_k) \in B$ has norm zero if and only
if at least one of its coordinates has norm zero.
 For $1\leq i_1< \cdots <i_j \leq k$, let $\Nm_{i_1,\cdots, i_j}(a)$
be the number of elements in the subalgebra
$$ B_{i_1, \cdots, i_j} = M_{d_{i_1}}(\F_{q^{n_{i_1}}}) \times \cdots \times M_{d_{i_j}}(\F_{q^{n_{i_j}}}).$$
with trace $a$.
The surjection of the trace map  implies that
$$\Nm_{i_1,\cdots, i_j}(a) =q^{n_{i_1}d_{i_1}^2+\cdots +n_{i_j}d_{i_j}^2 -1}.$$
By inclusion-exclusion, one finds that
$$\Nm_B(a, 0) = \sum_{j=1}^{k-1}(-1)^{k-1-j} \sum_{ 1\leq i_1< \cdots <i_j \leq k} q^{n_{i_1}d_{i_1}^2+\cdots +n_{i_j}d_{i_j}^2 -1} + (-1)^{k-1}\delta_{a,0},$$
where $\delta_{a, 0} =1$ if $a=0$, and $\delta_{a, 0} =0$ if $a \not= 0$.
This expression simplifies to the following explicit formula.

\begin{prop} Let $B$ be a finite semi-simple algebra over $\F_q$ of degree $n=\sum_{i=1}^k d_i^2 n_i \geq 2$ as given in equation (\ref{B}).
For  $a \in \F_q$, we have
$$\Nm_B(a, 0) = \frac{q^n - \prod_{j=1}^k(q^{n_jd_j^2}-1) +(-1)^k}{q} + (-1)^{k-1} \delta_{a, 0}.$$
\end{prop}

From now on, we assume that $b\not=0$.
Let $\psi$ be a fixed non-trivial additive character of $\F_q$.
For a multiplicative character of $\F_q^*$,  let $G(\chi,\psi)$ denote the Gauss sum
$$
G(\chi,\psi)=\sum_{x\in\F_q^*}\psi(x)\chi(x).
$$
Note that $G(\chi,\psi) =-1$ if $\chi$ is the trivial character, and $|G(\chi,\psi)|=\sqrt{q}$ if $\chi$ is non-trivial.
Let $\psi_{n_i}=\psi\circ {\Tr_{n_i}}$ be the lifted additive character of
$\F_{q^{n_i}}$, and let $\chi_{n_i} = \chi \circ {{\rm N}_{n_i}}$ be the lifted multiplicative character of $\F_{q^{n_i}}^*$. The Hasse-Davenport
relation gives
$$G(\chi_{n_i}, \psi_{n_i} ) =  \sum_{x\in\F_{q^{n_i}}^*}\psi_{n_i}(x)\chi_{n_i}(x) = (-1)^{n_i-1}G(\chi,\psi).$$
Now, we relate $\Nm_{B}(a,b)$ to algebraic exponential sums over $\F_q$. If $x=(x_1, \cdots, x_k) \in B$ with $\mathrm{N}_{B}(x)=b\not =0$, then $x$ is invertible and thus each $x_i$ is
invertible. Let $m=\sum_i d_in_i$.
Let $x_i$ below run over ${\rm GL}_{d_i}(\F_{q^{n_i}})$, and let $\chi$ run over all multiplicative characters of $\F_q^*$. By orthogonality of characters, one computes
\begin{align}\label{eq1}
        q(q-1) \Nm_B(a, b)
        &=\sum_{\left(x_1, \cdots, x_k\right)} \sum_{v \in \F_q} \psi\left(v(\sum_{i=1}^{k}\Tr_{n_i}\left({\rm tr}(x_i)\right)-a)\right)
        \sum_{\chi} \chi\left(b^{-1} \prod_{i=1}^{k}\Nm_{n_i}\left({\rm det}(x_i)\right)\right) \nonumber\\
        & =|B^*| + \sum_{v\neq 0} \psi(-a v) \sum_{\chi} \bar{\chi}(b) \sum_{\left(x_1, \cdots, x_k\right)} \prod_{i=1}^{k}\psi_{n_i}\left({\rm tr}(vx_i)\right)
        \prod_{i=1}^{k}\chi_{n_i}\left({\rm det}(x_i)\right) \nonumber\\
        & =|B^*| +\sum_{v \neq 0} \psi(-a v) \sum_{\chi} \bar{\chi}(bv^{\sum_i d_in_i})\prod_{i=1}^{k}\left(\sum_{x_i} \psi_{n_i}\left({\rm tr}(x_i)\right) \chi_{n_i}\left({\rm det}(x_i)\right)\right) \nonumber\\
&=|B^*| +\sum_{v \neq 0} \psi(-a v) \sum_{\chi} \bar{\chi}(bv^{m})\prod_{i=1}^{k} G_{{\rm GL}_{d_i}(\F_{q^{n_i}})}(\chi_{n_i}, \psi_{n_i}),
\end{align}
where \(G_{\mathrm{GL}_{d}(\mathbb{F}_{q})}(\chi,\psi)\)
stands for the Gauss sum over the general linear group ${\rm GL}_{d}(\F_{q})$ associated to the
character $(\chi, \psi)$ of $\F_{q}^* \times \F_{q}$:
\[
G_{\mathrm{GL}_{d}(\mathbb{F}_{q})}(\chi,\psi) =
\sum_{g\in \mathrm{GL}_{d}(\mathbb{F}_{q})}\chi(\det(g)) \psi(\mathrm{tr}(g)).
\]
By Eichler's formula \cite{EI} (see also \cite{La}\cite{Ki}\cite{LH} for new proofs),
 we have
$$G_{{\rm GL}_{d_i}(\F_{q^{n_i}})}(\chi_{n_i}, \psi_{n_i}) = \sum_{x_i\in {\rm GL}_{d_i}(\F_{q^{n_i})} }\psi_{n_i}\left({\rm tr}(x_i)\right) \chi_{n_i}\left({\rm det}(x_i)\right) =q^{n_i {d_i \choose 2}} G(\chi_{n_i}, \psi_{n_i})^{d_i}.$$
This together with the Hasse-Davenport relation gives
$$G_{{\rm GL}_{d_i}(\F_{q^{n_i}})}(\chi_{n_i}, \psi_{n_i})
 =(-1)^{d_i(n_i-1)}q^{n_i {d_i \choose 2}} G(\chi,\psi)^{d_in_i}.$$
Recall that $m=\sum_i d_in_i$ and $n= \sum_i d_i^2n_i$.
It follows that we have
\begin{align}
q(q-1) \Nm_B(a, b)&=|B^*| +(-1)^{\sum_i d_i(n_i-1)}q^{\sum_i n_i {d_i \choose 2}}\sum_{v \neq 0} \psi(-a v) \sum_{\chi} \bar{\chi}(bv^{m})
        G(\chi,\psi)^{\sum_i d_in_i} \nonumber\\
 &=|B^*| +(-1)^{m-\sum_i d_i} q^{\frac{n-m}{2}} S_m(a, b),
                \end{align}
                where
\begin{align}\label{sum}
S_m(a, b)&= \sum_{v \neq 0} \psi(-a v) \sum_{\chi} \bar{\chi}(bv^{m})
        G(\chi,\psi)^{m}
        \end{align}
depends only on the integer \(m\) and the elements \(a,b \in \mathbb{F}_{q}\).
        This gives
\begin{lemma}\label{lem10} We have the relation
\begin{align}\label{lem}
\Nm_B(a, b)&=\frac{|B^*|}{q(q-1)} + \frac{(-1)^{m-\sum_i d_i} q^{\frac{n-m}{2}}}{q(q-1)} S_m(a, b).
                \end{align}
\end{lemma}

Now, we are ready to prove Theorem \ref{thm}. More precisely, we have

\begin{theorem}\label{thm'} Let $B$ be a finite semi-simple algebra over $\F_q$ of degree $n=\sum_{i=1}^k d_i^2 n_i \geq 2$ as given in equation (\ref{B}).   Let $m=\sum_i d_in_i$. For $a \in \F_q$ and $b\in\F_q^*$, we have
        \begin{align*}
                \left| \Nm_B(a, b) -
                \left( \frac{|B^*|}{q(q-1)}
                +\frac{(-1)^{\sum_i d_i}q^{\frac{n-m}{2}}}{q} \right) \right|
                \leq (m -1)q^{\frac{n-2}{2}}.
        \end{align*}
        For $a=0$, we have the more precise elementary bound:
        \begin{align*}
                \left| \Nm_B(0, b) -
                \left( \frac{|B^*|}{q(q-1)}
                +\frac{(-1)^{\sum_i d_i}q^{\frac{n-m}{2}}}{q} \right) \right|
                \leq ((m, q-1)-1)q^{\frac{n-2}{2}}.
        \end{align*}
        For $a\not=0$, we also have the following elementary bound:
\begin{align*}
        \left|\Nm_B(a, b) - \frac{|B^*|}{q(q-1)}\right|
        & \leq  \frac{(m, q-1)}{q-1}q^{\frac{n-2}{2}} + \frac{(q-1-(m, q-1))}{q-1}q^{\frac{n-1}{2}} < q^{\frac{n-1}{2}}.
        \end{align*}

        \end{theorem}

\begin{proof}

Write
\begin{align}\label{sum2}
S_m(a, b) = (-1)^m(q-1) + T_m(a, b).
        \end{align}
Lemma \ref{lem10} can be rewritten as
\begin{align}
\Nm_B(a, b)&- \left(\frac{|B^*|}{q(q-1)} + \frac{(-1)^{\sum_i d_i} q^{\frac{n-m}{2}}}{q}\right) = \frac{(-1)^{m-\sum_i d_i} q^{\frac{n-m}{2}}}{q(q-1)} T_m(a, b).
                \end{align}
Specializing this general formula to the special case $B=\F_q^m$, we also have
\begin{align}
\Nm_{\F_q^m}(a, b)&- \left(\frac{(q-1)^{m-1}}{q} + \frac{(-1)^{m}}{q}\right) = \frac{1}{q(q-1)}T_m(a, b).
                \end{align}
Combine the two equations, we obtain
$$
                                       \Nm_B(a, b) -
                (\frac{|B^*|}{q(q-1)} +\frac{(-1)^{\sum_i d_i}q^{\frac{n-m}{2}}}{q})
                = (-1)^{m-\sum_i d_i} q^{\frac{n-m}{2}}(\Nm_{\F_q^m}(a, b) -
                \frac{(q-1)^{m-1} +(-1)^m}{q} ).
                $$
                This proves the reduction lemma \ref{lem0} in the introduction.
This  and the estimate (\ref{imp'}) together prove the first part of the theorem.

If $a=0$, equation (\ref{sum}) gives
\begin{align}\label{eq3}
        S_m(0, b)
        & = \sum_{\chi} \sum_{v\in \F_q^*} \bar{\chi}^m(v) G(\chi, \psi)^{m} \bar{\chi}(b) \nonumber\\
        & =(q-1)  \sum_{\chi^m=1} G(\chi, \psi)^{m} \bar{\chi}(b)\nonumber\\
        &=(q-1)\left((-1)^{m} +\sum_{\chi^{m}=1, \chi\not=1} G(\chi, \psi)^{m} \cdot \bar{\chi}(b)\right).
\end{align}
Since $|G(\chi, \psi)| =\sqrt{q}$, using Lemma \ref{lem10}, we deduce that
\begin{align*}
                \left| \Nm_{\F_q^m}(0, b) -
                \left( \frac{|B^*|}{q(q-1)}
                +\frac{(-1)^{\sum_i d_i} q^{\frac{n-m}{2}}}{q} \right) \right|
                \leq ((m, q-1)-1)q^{\frac{n-2}{2}}.
        \end{align*}
This proves the second part of Theorem \ref{thm'}.

If $a\not=0$,  equation (\ref{sum}) can be written as
\begin{align}\label{eqe}
        S_m(a, b)
 &=\sum_{\chi} {\chi}(\frac{(-a)^{m}}{b})
        G(\bar{\chi}^{m}, \psi)G(\chi, \psi)^{m}.
        \end{align}
Separating the case $\chi^{m}=1$, we obtain the bound
\begin{align*}
        \left|S_m(a, b) \right|	& \leq (m, q-1)q^{\frac{m}{2}} + (q-1-(m, q-1))q^{\frac{m+1}{2}}.
        \end{align*}
This and Lemma \ref{lem10} give
\begin{align*}
        \left|\Nm_B(a, b) - \frac{|B^*|}{q(q-1)}\right|
        & \leq \frac{(m, q-1)}{q-1}q^{\frac{n-2}{2}} + \frac{(q-1-(m, q-1))}{q-1}q^{\frac{n-1}{2}} < q^{\frac{n-1}{2}}.
        \end{align*}
This proves the third part of Theorem \ref{thm'}.
\end{proof}

A much easier question is to count the number of elements in $B^*$ with fixed trace $a \in \F_q$. Let
$$\Nm_{B^*}(a): = \{ x \in B^* \ | \ {\rm Tr}_B(x) = a\} = \sum_{b\in \F_q^*} \Nm_B(a, b).$$
By Lemma \ref{lem10}, we have
$$\Nm_{B^*}(a)=\frac{|B^*|}{q} + \frac{(-1)^{m-\sum_i d_i} q^{\frac{n-m}{2}}}{q(q-1)} \sum_{b \in \F_q^*}S_m(a, b). $$
By equation (\ref{sum}), one computes
$$\sum_{b\in \F_q^*} S_m(a, b)= \sum_{v \neq 0} \psi(-a v) \sum_{\chi} \sum_{b\in \F_q^*}\bar{\chi}(bv^{m})
        G(\chi, \psi)^{m} = (-1)^m(q-1)\sum_{v \neq 0} \psi(-a v).$$
        This number is $(-1)^{m+1}(q-1)$ if $a\not=0$, and $(-1)^m(q-1)^2$ if $a=0$.
        It follows that we have the following formula for $\Nm_{B^*}(a)$.


\begin{prop}\label{prop-tr} Let $B$ be a finite semi-simple algebra over $\F_q$ of degree $n=\sum_{i=1}^k d_i^2 n_i \geq 2$ as given in equation (\ref{B}).   Let $m=\sum_i d_in_i$ and $a \in \F_q$.
        For $a=0$, we have
        \begin{align*}
                \Nm_{B^*}(0)  = \frac{1}{q}
                \left( |B^*| + (-1)^{\sum_i d_i}(q-1)q^{\frac{n-m}{2}}\right).
        \end{align*}
        For $a\not=0$, we have
\begin{align*}
\Nm_{B^*}(a)  = \frac{1}{q}\left( |B^*| - (-1)^{\sum_i d_i}q^{\frac{n-m}{2}}\right).
\end{align*}

        \end{prop}
In the case $B= M_{d}(\F_q)$, this recovers Kim \cite[Theorem 6.2]{Ki} and Li-Hu \cite[Theorem 3.1]{LH}.

\begin{remark}\label{prob1} Given a polynomial $f(t)\in \F_q[t]$ of degree $r\geq 1$ with $(r, p)=1$.
A harder problem is to count the number of elements $x \in B$
such that ${\rm Tr}_B(f(x))=a \in \F_q$. Namely, for $a \in \F_q$, let
$$\Nm_{\{B, f\}}(a): = \# \{ x \in B \ | \ {\rm Tr}_B(f(x)) = a\}.
$$
 In the geometric case $B =\F_q^n$, it is clear that
 $$\Nm_{\{\F_q^n, f\}}(a): = \# \{ (x_1, \cdots, x_n) \in \F_q^n \ | \ f(x_1)+ \cdots +f(x_n) = a\}.
$$
If the affine hypersurface $f(x_1)+ \cdots +f(x_n) = a$ is smooth (its infinite part is automatically smooth since $(r, p)=1$), then Deligne's well known
theorem gives the square root estimate
\begin{align}\label{01}
\left| \Nm_{\{\F_q^n, f\}}(a) - q^{n-1}\right| \leq (r-1)^n q^{\frac{n-1}{2}}.
\end{align}
The smooth condition cannot be dropped, for instance, the estimate (\ref{01}) fails if $f(x)=x^r$ with $r>1$ and $a=0$.

In the non-split case $B = \F_{q^n}$, the number $\Nm_{\{\F_{q^n}, f\}}(a)$ is $1/q$ times the number of $\F_{q^n}$-rational points on the Artin-Schreier
curve $y^q - y = f(x) - \alpha$, where $\alpha$ is any element in $\F_{q^n}$ such that ${\rm Tr}_n(\alpha) = a$. The well known Weil bound gives
\begin{align}\label{1}
\left| \Nm_{\{\F_{q^n}, f\}}(a) - q^{n-1}\right| \leq (r-1)(q-1) q^{\frac{n-2}{2}} \leq (r-1)q^{\frac{n}{2}}.
\end{align}
This is not a square root estimate for large $q$.
The number $\Nm_{\{\F_{q^n}, f\}}(a)$ is systematically studied in Rojas-Leon etc \cite{RW11, RL12, RL13} using Deligne's theorem,
proving results which are square root estimates and hence stronger than the Weil
bound (\ref{1}) when $q$ is large compared with $r$. Roughly speaking, one gets an extra $\sqrt{q}$ saving for the error term
if $f(t)$ is square-free and either $r$ is odd or the hypersurface $f(x_1)+\cdots +f(x_n)-a=0$ is smooth.
It would be interesting to extend these results to $\Nm_{\{B, f\}}(a)$ for
general semi-simple algebra $B$ over $\F_q$.
 \end{remark}

 \begin{remark}
 For $a \in \F_q$ and $b \in \F_q^*$, one can similarly define
$$\Nm_{\{B, f\}}(a, b): = \{ x \in B^* \ | \ {\rm Tr}_B(f(x)) = a, \ {\rm N}_B(x)=b\}.
$$
The main result of this paper, Theorem \ref{thm}, gives a uniform square root estimate for $\Nm_{\{B, f\}}(a, b)$ in the case $f(t)$ is a linear polynomial. It would be   interesting to extend this uniform estimate to higher degree polynomial $f(t)$. In the geometric case $B = \F_q^n$, the number
$\Nm_{\{\F_q^n, f\}}(a, b)$ is the number of $\F_q$-rational points on the following toric hypersurface
$$f(x_1) + \cdots + f(x_{n-1}) + f\bigg(\frac{b}{x_1\cdots x_{n-1}}\bigg) - a = 0.$$
The subtlety of the problem is that this hypersurface is singular for some $a$. So, it is probably
unreasonable to expect a square root estimate, which is uniform in $a \in \F_q$.
In some non-trivial cases, one can expect such uniform estimate for all non-zero $a\in \F_q^*$.
As mentioned before, the answer is yes if $r=\deg(f)=1$.
\end{remark}

\section{Proof of Theorem \ref{K}}

Let $B$ be a finite semi-simple algebra over $\F_q$ of degree $n=\sum_i d_i^2n_i\geq 2$.
Let $\psi: \F_q \longrightarrow \C$ be a non-trivial additive character.
For $b\in \F_q^*$,  recall that the Kloosterman sum over $B$ is defined by
$$K_{B}(b): = \sum_{x \in B^*, \ \mathrm{N}_{B}(x)=b} \psi\left({\rm Tr}_B(x)\right).$$
Similar to the trace-norm reduction in Lemma \ref{lem0}, we have the following reduction for exponential sum $K_B(b)$.
\begin{lemma}\label{lem2}
Let $B$ be a finite semi-simple algebra over $\F_q$ of degree $n=\sum_{i=1}^k d_i^2 n_i \geq 2$ as given in equation (\ref{B}) with $m=\sum_{i=1}^k d_in_i$.
For $b\in\F_q^*$, we have
\begin{align}\label{kl5}
K_{B}(b) =  (-1)^{m - \sum_i d_i}q^{\frac{n-m}{2}} K_{\F_q^m}(b).
                \end{align}
                \end{lemma}

\begin{proof}
By definition, we can rewrite
$$K_{B}(b)= \sum_{a\in \F_q} \Nm_B(a,b)\psi(a).$$
By Lemma \ref{lem10}, one computes
\begin{align*}
K_{B}(b) &= \sum_{a\in \F_q} \left( \frac{|B^*|}{q(q-1)}+ \frac{ (-1)^{m - \sum_i d_i}q^{\frac{n-m}{2}}}{q(q-1)} S_m(a, b)\right) \psi(a) \nonumber\\
        &=  \frac{(-1)^{m - \sum_i d_i}q^{\frac{n-m}{2}}}{q(q-1)} \sum_{v \in \F_q^*}\sum_{a \in \F_q} \psi( a-av) \sum_{\chi} \bar{\chi}(bv^m)
        G(\chi,\psi)^{m} \nonumber \\
        &=  \frac{(-1)^{m - \sum_i d_i}q^{\frac{n-m}{2}}}{q-1}\sum_{\chi} \bar{\chi}(b)
        G(\chi,\psi)^{m} \nonumber \\
        &=  \frac{(-1)^{m - \sum_i d_i}q^{\frac{n-m}{2}}}{q-1}\sum_{y_1, \cdots, y_m \in \F_q^*}
        \sum_{\chi} \chi\bigg(\frac{y_1\cdots y_m}{b}\bigg) \psi(y_1+\cdots +y_m) \nonumber \\
        &=  (-1)^{m - \sum_i d_i}q^{\frac{n-m}{2}}\sum_{y_i \in \F_q^*, \ y_1\cdots y_m = b}\psi(y_1+\cdots +y_m) \nonumber \\
        &=  (-1)^{m - \sum_i d_i}q^{\frac{n-m}{2}} K_{\F_q^m}(b).\qedhere
                \end{align*}
                \end{proof}
Note that $K_{\F_q^m}(b)$ is the classical hyper-Kloosterman sum in (\ref{K-clas}). Deligne's theorem gives the
bound $|K_{\F_q^m}(b)| \leq m q^{(m-1)/2}$. It follows that
$$|K_B(b)| \leq m q^{\frac{n-m}{2}} q^{\frac{m-1}{2}} = m q^{\frac{n-1}{2}}.$$
Theorem \ref{K} is proved.

\begin{remark}
For a finite semi-simple algebra $B$ over $\F_q$ of degree $n=\sum_{i=1}^k d_i^2 n_i $ as given in equation (\ref{B}) with $m=\sum_{i=1}^k d_in_i$,
one can consider the much simpler full sum
$$K_{B^*}: = \sum_{x \in B^*} \psi({\rm Tr}_B(x)) = \sum_{a\in \F_q} \Nm_{B^*}(a)\psi(a).$$
By Proposition \ref{prop-tr}, we obtain
$$K_{B^*} = \frac{1}{q}\left( |B^*| - (-1)^{\sum_i d_i} q^{\frac{n-m}{2}}\right) \sum_{a\in \F_q} \psi(a) + (-1)^{\sum_i d_i} q^{\frac{n-m}{2}} = (-1)^{\sum_i d_i} q^{\frac{n-m}{2}}.$$
\end{remark}

\begin{remark} A much more interesting problem is to estimate the following higher degree sum
\begin{align}
K_{\{B,f\}}(b): = \sum_{x \in B^*, \ {\rm N}_{B}(x)=b} \psi\left({\rm Tr}_B(f(x))\right), \ b \in \F_q^*,
\end{align}
where $f(t)\in \F_q[t]$ is a polynomial of degree $r$ not divisible by $p$. When $f(t)$ is a linear polynomial,
the sum is reduced to the Kloosterman sum over $B$ studied in Theorem \ref{K}. It would be interesting to prove
a similar square root estimate for $K_{\{B,f\}}(b)$ for higher degree polynomial $f(t)$. For non-commutative
semi-simple algebra $B$ over $\F_q$, we do not know a good estimate for $K_{\{B,f\}}(b)$.

\end{remark}

\begin{remark} We now explain how to obtain a square root estimate for the higher degree sum $K_{\{B,f\}}(b)$ when $B$ is a finite \'etale agebra over $\F_q$,
where $f(t)\in \F_q[t]$ is any polynomial of degree $r$ not divisible by $p$.

In the split case $B =\F_q^n$,  applying the work of
Adolphson-Sperber \cite[Theorem 4.2]{AS} and Denef-Loeser \cite{DL} for toric exponential sums, one obtains the following estimate:

\begin{align}\left| K_{\{\F_q^n,f\}}(b): = \sum_{x_1, \cdots, x_n \in {\F_q^*}, \ x_1\cdots x_n=b} \psi\left(f(x_1)+\cdots + f(x_n)\right)\right| \leq nr^{n-1} q^{\frac{n-1}{2}}, \ b\in \F_q^*.
\end{align}
This is because the sum $ K_{\{\F_q^n,f\}}(b)$ becomes the toric exponential sum of the following $(n-1)$-variable Laurent polynomial
$$f(x_1)+\cdots + f(x_{n-1}) + f\bigg(\frac{b}{x_1\cdots x_{n-1}}\bigg)$$
which is easily seen to be non-degenerate with respect to its Newton polytope in $\mathbb{R}^{n-1}$.

In the general finite \'etale case $B= \prod_{i=1}^k \F_{q^{n_i}}$ of degree $n=\sum n_i$, the same estimate
\begin{align}\label{es1}
\left| K_{\{B,f\}}(b): = \sum_{x \in B^*, \ {\rm Norm}(x)=b} \psi\left({\rm Tr}_B(f(x))\right)\right| \leq n r^{n-1} q^{\frac{n-1}{2}}, \ b \in \F_q^*
\end{align}
holds by a twisting or Weil descent argument which essentially reduce the problem to the split case.
This is already a good estimate.
\end{remark}

\begin{remark}
The estimate in (\ref{es1}) can be improved in some cases.  In the special case when $B= \F_{q^n}$ is a field,  taking $\beta \in \F_{q^n}^*$ with ${\rm N}_n(\beta)=b$, one finds that the sum $K_{\{B,f\}}(b)$ becomes the
following one variable sum over $\F_{q^n}^*$:
\begin{align}
K_{\{\F_{q^n},f\}}(b): = \frac{1}{q-1}\sum_{x \in \F_{q^n}^*} \psi\left({\rm Tr}_n (f(\beta x^{q-1}))\right).
\end{align}
The Weil bound for one variable exponential sum over $\F_{q^n}^*$ gives the estimate
\begin{align}\label{es5}
\left| K_{\{\F_{q^n},f\}}(b) \right| \leq r q^{\frac{n}{2}}.
\end{align}
This is not a square root estimate. It is better than the estimate (\ref{es1}) if $q < n^2 r^{2n-2}$, but worse than the estimate (\ref{es1}) if $q > n^2 r^{2n-2}$.
Roughly speaking, for large $q$,  the estimate (\ref{es1}) gives an extra $\sqrt{q}$ saving over the Weil bound (\ref{es5}).
Remarkably, for large $q$ and most $f$, the following even stronger estimate is
obtained in Rojas-Leon \cite[Example 7.22]{RL12}:
\begin{align}
\left|K_{\{\F_{q^n},f\}}(b): = \frac{1}{q-1}\sum_{x \in \F_{q^n}^*} \psi\left({\rm Tr}_n (f(\beta x^{q-1}))\right) \right| \leq C_{r,n} q^{\frac{n-1}{2}}, \ b \in \F_q^*,
\end{align}
where $f(t)$ is assumed not to be of the form $g(t^e)$ for $e>1$ and $g(t) \in \F_q[t]$.
The constant $C_{r, n}$ \cite[Remark 7.23]{RL12}\cite[Remark 8]{RL13b} is given by
\begin{align}
C_{r, n} = \frac{n}{r} \sum_{i=0}^{n-1}{n+r-i-1\choose n}{n-1\choose i}. 
\end{align}
The constant $C_{r, n}$ is significantly smaller than the constant $ n r^{n-1}$ for $r>1$.
The theory developed in Rojas-Leon \cite{RL12}
is applicable to the general commutative case when $B$ is a finite \'etale algebra over $\F_q$.
It would be interesting to work out such improvements for general finite \'etale algebra $B$ over $\F_q$,
or even explore the non-commutative case when $B$ is a finite semi-simple algebra over $\F_q$.
\end{remark}

\section{Product-trace version}

Let $B$ be a semi-simple algebra of dimension $n$ over $\F_q$ as before. We now consider the product-trace version of the Kloosterman sum and its counting analogue over $B$. This is motivated by the matrix Kloosterman sum studied
in Zelingher \cite{Ze}.

Let $\psi$ be a non-trivial additive character of $\F_q$.
Given $x \in B^*$ and integer $r\geq 2$.
Define the product-trace Kloosterman sum over $B$ by
$$K(B, r, x, \psi):=\sum_{g_1\cdots g_r=x, ~g_i \in B^*} \psi ({\rm Tr}_B(g_1+\cdots +g_r)).$$
In the case $B=\F_q$, this sum is exactly the classical hyper-Kloosterman sum.
It is natural to expect that there is a square root cancellation in the sum
$K(B, r, x, \psi)$ for general $B$. This is not true for arbitrary $x \in B^*$, but turns out to be true for regular element
$x \in B^*$ as defined below.

\begin{definition}
An element $x \in {\rm GL}_d(\bar{\F}_q)$ is called regular if the number of
independent eigenvectors of $x$ over $\bar{\F}_q$ is equal to the number of distinct eigenvalues
of $x$ over $\bar{\F}_q$.  An element $x=(x_1, \cdots, x_k) \in B^* = {\rm GL}_{d_1}(\F_{q^{n_1}}) \times \cdots \times {\rm GL}_{d_k}(\F_{q^{n_k}})$ is called regular
if each $x_i \in {\rm GL}_{d_i}(\F_{q^{n_i}})$ is regular.

\end{definition}

The results in \cite{Ze} imply the following

\begin{theorem}\label{conj1} Let $B = M_{d_1}(\F_{q^{n_1}}) \times \cdots \times M_{d_k}(\F_{q^{n_k}})$ be a semi-simple algebra of dimension $n$ over $\F_q$. Let $x=(x_1, \cdots, x_k) \in B^*$ be regular. Then for integer $r\geq 2$, we have
$$|K(B, r, x, \psi)|  \leq r^{d_1+\cdots +d_k} q^{\frac{(r-1)n}{2}}.$$
\end{theorem}

\begin{proof}
In the \'etale case $B=\prod_{i=1}^k \F_{q^{n_i}}$,
each $x=(x_1, \cdots, x_k) \in B^* = \prod_{i=1}^k \F_{q^{n_i}}^*$ is automatically
regular. In this case, one checks that
$$K(B, r, x, \psi)=\prod_{i=1}^k \sum_{y_{i1}, \cdots, y_{i(r-1)} \in \F_{q^{n_i}}^*}
\psi \left({\rm Tr}_{n_i}(y_{i1}+\cdots +y_{i(r-1)} +\frac{x_i}{y_{i1}\cdots y_{i(r-1)}}
)\right).$$
By Deligne's theorem, we deduce
$$|K(B, r, x, \psi)| \leq r^k q^{\frac{(r-1)(n_1+\cdots +n_k)}{2}} = r^k q^{\frac{(r-1)n}{2}}.$$
Thus, Theorem \ref{conj1} holds when $B$ is \'etale.

In the matrix algebra case $B= M_{d}(\F_q)$, the sum $K(M_{d}(\F_q), r, x, \psi)$ is exactly the matrix Kloosterman sum studied in \cite{Ze}.
Let $m_x(t)$ be the characteristic polynomial of $x \in {\rm GL}_d(\F_q)$. Let
$$m_x(t) = \prod_{i=1}^s f_i(t)^{b_i}$$
be the standard factorization of $m_x(t)$ over $\F_q$, where each $b_i\geq 1$, each
$f_i(t)$ is a monic  irreducible polynomial over $\F_q$ of degree $a_i\geq 1$, and the
$f_i(t)$'s are distinct. Since $x\in {\rm GL}_d(\F_q)$ is an regular element, Zelingher \cite[Corollary 4.9]{Ze} gives the following estimate
$$|K(M_d(\F_q), r, x, \psi)| \leq q^{\frac{(r-1)d^2}{2}}\prod_{j=1}^s{b_j+r-1\choose b_j}.$$
Now, for integer $b\geq 1$, 
$${b+r-1\choose b} = \frac{b+r-1}{b} \frac{b-1+r-1}{b-1}\cdots \frac{r}{1} \leq r^b.$$
Since $b_1+\cdots + b_s \leq b_1a_1+\cdots +b_sa_s = d$, it follows that
$$|K(M_d(\F_q), r, x, \psi)| \leq q^{\frac{(r-1)d^2}{2}}\prod_{j=1}^s r^{b_j} \leq r^{d} q^{\frac{(r-1)d^2}{2}}.$$
The theorem is true when $B= M_{d}(\F_q)$.

If now $B= M_{d}(\F_{q^m})$ as $\F_q$-algebra, we let $B_m= B$ as $\F_{q^m}$-algebra
and $\psi_m = \psi \circ {\rm Tr}_{\F_{q^m}}/\F_q$ (a non-trivial additive character of $\F_{q^m}$), then one checks that
$K(B, r, x, \psi) = K(B_m, r, x, \psi_m)$. It follows that for regular $x \in {\rm GL}_d(\F_{q^m})$,
$$|K(B, r, x, \psi)| = |K(B_m, r, x, \psi_m)| \leq  r^{d}q^{\frac{(r-1)d^2m}{2}}.$$
In the general case
$$B = M_{d_1}(\F_{q^{n_1}}) \times \cdots \times M_{d_k}(\F_{q^{n_k}}),$$
for $x=(x_1, \cdots, x_k) \in B^*$ with $x_i \in {\rm GL}_{d_i}(\F_{q^{n_i}})$,
one checks that
$$K(B, r, x, \psi) = \prod_{i=1}^k K(M_{d_i}(\F_{q^{n_i}}), r, x_i, \psi).$$
Since $x$ is regular, each $x_i$ is regular. By the above matrix algebra case, we conclude that
$$|K(B, r, x, \psi)| \leq \prod_{i=1}^k (r^{d_i}q^{\frac{(r-1)d_i^2n_i}{2}}) \leq r^{d_1+\cdots +d_k} q^{\frac{(r-1)n}{2}}.$$
The theorem is proved.
\end{proof}

In view of the norm-trace estimate in Theoreom \ref{thm}, it would be interesting to study the following counting version of the product-trace problem. For $x \in B^*$, $a \in \F_q$ and $r\geq 2$,
define
$$N(B, r, x,a): = \#\{ (g_1,\cdots, g_r) \in (B^*)^r ~|~ g_1\cdots g_r = x, ~ {\rm Tr}_B(g_1+\cdots +g_r) =a\}.$$
Heuristically,
the number $N(B, r, x, a)$ should be roughly $|B^*|^{r-1}/q$,  with hopefully a square root error term for regular $x\in B^*$. Namely, one hopes that the following holds.

\begin{conjecture}\label{conj2} Let $B$ be a semi-simple algebra of dimension $n$ over $\F_q$. For $a\in \F_q^*$,
regular $x \in B^*$ and integer $r\geq 2$, we have
$$\left|N(B, r, x, a) - \frac{|B^*|^{r-1}}{q}\right| \leq r^{d_1+\dots +d_k} q^{\frac{(r-1)n-1}{2}}.$$
\end{conjecture}
Note that the conjecture can fail in some cases when $a=0$. That is why we assume that $a \in \F_q^*$
is non-zero in the conjecture.

In the geometric case when $B=\F_q^n$, for $x=(x_1, \cdots, x_n) \in (\F_q^*)^n$ and $a\in \F_q^*$, we have
$$N(\F_q^n, r, x,a)= \#\left\{ y_{ij} \in \F_q^*, 1\leq i\leq r-1, 1\leq j\leq  n~|~ \sum_{j=1}^n\bigg(\sum_{i=1}^{r-1}y_{ij}+ \frac{x_j}{y_{ij}\cdots y_{(r-1)j}}\bigg)=a\right\}.$$
The method of $\ell$-adic cohomology and Kloosterman sheaf as explained in \cite{LW} 
can be generalized to prove
$$|N(\F_q^n, r, x, a) -\frac{(q-1)^{n(r-1)}}{q} | \leq r^n q^{\frac{(r-1)n-1}{2}}.$$
This shows that Conjecture \ref{conj2} is true in the geometric case when $B=\F_q^n$.

When $B=\F_{q^n}$ (a field extension), for $x \in \F_{q^n}^*$, we have
$$N(\F_{q^n}, r, x,a)= \#\left\{ (y_1, \cdots, y_{r-1}) \in (\F_{q^n}^*)^{r-1}~|~
{\rm Tr}_{\F_{q^n}/\F_q}\bigg(y_1+\cdots +y_{r-1}+  \frac{x}{y_{1}\cdots y_{r-1}}\bigg)=a\right\}.$$
It is not clear how to handle this for all $a\in \F_q^*$. The Weil descent argument can probably be used to
show that the square root estimate is true for most $a\in \F_q^*$ depending on $x$.
Here, we are hoping that the conjecture
is true uniformly for all non-zero $a\in \F_q^*$.

The reduction method in Zelingher \cite{Ze} via representation theory should reduce
the general non-commutative case to the commutative case with $B$ being \'etale. So, the field extension case when $B=\F_{q^n}$ seems to be
the most essential case.

\section*{Acknowledgements} It is a pleasure to thank E. Zelingher and Dingxin Zhang for discussions
on this topic, and the referee for helpful comments.


\begin{thebibliography}{99}

\bibitem{AS}
\newblock A. Adolphson and S. Sperber, 
\newblock Exponential sums and Newton polyhedra: cohomology and estimates,
\newblock \emph{Ann. Math.}, \textbf{130} (1989), 367-406.

\bibitem{EI} 
\newblock M.  Eichler, 
\newblock Allgemeine Kongruenz-Klasseneinteilungen der Ideale einfacher Algebren
        \"uber algebraischen Zahlk\"orpern und ihre L-Reihen. 
\newblock \emph{J. Reine Angew. Math.}, \textbf{179} (1937), 227 - 251.



\bibitem{MMM} 
\newblock G. L. Matthews, T.  Morrison, and A. W. Murphy, 
\newblock Curve-lifted codes for local recovery using lines, 
\newblock\emph{Des. Codes Cryptogr.}, \textbf{92} (2024), 3645-3664.



\bibitem{De}
\newblock P. Deligne, 
\newblock “Cohomologie \'{e}tale (SGA 4 $\frac{1}{2}$)", 
\newblock \emph{Lecture Notes in Mathematics}, Vol. 569,  Springer-Verlag, Berlin/Heidelberg/New York, 1977.


\bibitem{DL}
\newblock J. Denef and F. Loeser, 
\newblock Weights of exponential sums, intersection cohomology, and Newton polyhedra,
 \newblock  \emph{Invent. Math.}, Vol. \textbf{106} (1991), 275-294.


\bibitem{K0} 
\newblock N. Katz, Gauss Sums, 
\newblock \emph{Kloosterman Sums, and Monodromy Groups},
\newblock Princeton University Press, 1988.


\bibitem{Ka} 
\newblock N. Katz, 
\newblock Estimates for Soto-Andrade sums, 
\newblock \emph{J. Reine Angew. Math.}, \textbf{438} (1993), 143-161.


\bibitem{Ki} 
\newblock D. Kim, 
\newblock Gauss sums for general and special linear groups over a finite field,
\newblock \emph{Arch. Math.}, \textbf{69} (1997),  297-304.

\bibitem{Ki11} 
\newblock D. Kim, 
\newblock Codes associated with special linear groups and power moments of multi-dimensional Kloosterman sums,
\newblock  \emph{Annali di Matematica}, \textbf{190} (2011), 61-76.


 \bibitem{La} 
\newblock E. Lamprecht, 
\newblock Struktur und Relationen allgemeiner Gaußcher Summen in endlichen Ringen I, II,
\newblock  \emph{J. Reine Angew. Math.},  \textbf{197} (1957),  1-48.

        
\bibitem{LH} 
\newblock Y. Li and S. Hu, 
\newblock Gauss sums over some matrix groups, 
\newblock \emph{J.  Number Theory}, \textbf{132} (2012), 2967-2976.


\bibitem{LW} 
\newblock X. Lin and D. Wan, 
\newblock Counting elements with given trace and norm in étale algebras, 
\newblock \emph{International  J. Number Theory}, \textbf{21} (2025), 1955-1965.


\bibitem{Mo}
\newblock M. Moisio, 
\newblock Kloosterman sums, elliptic curves, and irreducible polynomials with prescribed trace and norm, 
\newblock \emph{Acta Arith.}, \textbf{132} (2008), 329-350.

        
\bibitem{MW10}
\newblock M. Moisio and D. Wan, 
\newblock On Katz's bound for the number of elements with given trace and norm, 
\newblock \emph{J. Reine Angew. Math.}, \textbf{638} (2010), 69-74.


\bibitem{RL12}
\newblock A. Rojas-Leon, 
\newblock Rationality of trace and norm $L$-functions,
 \newblock  \emph{Duke Math. J.}, \textbf{161} (2012),  1751-1795.

        
\bibitem{RL13}
\newblock A. Rojas-Leon, 
\newblock On the number of rational points on curves over finite fields with many automorphisms,
 \newblock       \emph{Finite Fields \& Appl.}, \textbf{19} (2013),  1-15.

\bibitem{RL13b}
\newblock A. Rojas-Leon, 
\newblock Local convolution of $\ell$-adic sheaves on the torus, 
 \newblock       \emph{Math Z.}, \textbf{274} (2013),  1211-1230. 

       
\bibitem{RW07}
\newblock A. Rojas-Leon and D. Wan,
\newblock  Moment zeta functions for toric Calabi-Yau hypersurfaces,
\newblock        \emph{Commun. Number Th. Phys.}, \textbf{1} (2007), 539-578.

        
\bibitem{RW11}
\newblock A. Rojas-Leon and D. Wan, 
\newblock Improvements of the Weil bound for Artin-Schreier curves,
 \newblock       \emph{Math. Ann.}, \textbf{351} (2011), 417–442.


\bibitem{Wa}
\newblock D. Wan, 
\newblock Lectures on zeta functions over finite fields, 
\newblock \emph{Higher Dimensional Geometry over Finite Fields}, D. Kaledin and Y. Tschinkel, eds., ISO Press (2008), 244-268.

        
\bibitem{Ze}
\newblock E. Zelingher, 
\newblock On matrix Kloosterman sums and Hall-Littlewood polynomials,
\newblock \emph{Trans. Amer. Math. Soc.}, \textbf{378} (2025), 3597–3623.




\end{thebibliography}
\end{document}